\documentclass[a4paper]{scrartcl}
\usepackage{amsmath}
\usepackage{amsthm}
\usepackage{amssymb}
\usepackage{xypic}
\usepackage{graphicx}

\newtheorem{theorem}{Theorem}[section]
\newtheorem{lemma}[theorem]{Lemma}

\newtheorem{cor}[theorem]{Corollary}
\newtheorem{prop}[theorem]{Proposition}

\theoremstyle{definition}
\newtheorem{definition}[theorem]{Definition}

\theoremstyle{remark}
\newtheorem{remark}[theorem]{Remark}

\numberwithin{equation}{section}

\newcommand{\N}{\mathbb{N}}
\newcommand{\R}{\mathbb{R}}

\DeclareMathOperator{\Iso}{Iso}
\DeclareMathOperator{\Diff}{Diff}
\DeclareMathOperator{\Id}{Id}
\DeclareMathOperator{\diam}{diam}

\DeclareMathOperator{\Ric}{Ric}

\DeclareMathOperator{\seccur}{sec}

\title{On the topology of moduli spaces of non-negatively curved Riemannian metrics}

\author{Wilderich Tuschmann\thanks{W.T. acknowledges support from the HeKKSaGon University Network and DFG Priority Program SPP 2026 'Geometry at Infinity'.} \and Michael Wiemeler\thanks{M.W. acknowledges support by DFG Grant HA 3160/11-1 and SFB 878.}}

\date{}

\begin{document}

\maketitle

\begin{abstract}
 We study spaces and moduli spaces of Riemannian metrics with non-negative Ricci or
 non-negative sectional curvature on closed and open manifolds.
 We construct, in particular, the first classes of  manifolds for which these moduli spaces have non-trivial rational homotopy,
 homology and cohomology groups.
 We also show that in every dimension at least seven
(respectively, at least eight) there exist infinite sequences of closed (respectively, open) manifolds of pairwise distinct homotopy type
for which the space and moduli space  of Riemannian metrics with non-negative sectional curvature 
has infinitely many path components. A completely analogous statement holds for spaces and  moduli 
spaces of
non-negative Ricci curvature metrics.
\end{abstract}



\section{Introduction}
\label{sec:introduction}

Consider a smooth manifold $M$ with a Riemannian metric satisfying some sort of geometric constraints
like, for example, having positive scalar curvature, non-negative Ricci or sectional curvature, negative sectional curvature, 
being Einstein, K\"ahler, Sasaki, etc.
A natural question to ponder is what the space of all such metrics on $M$ looks like.
One can also ask a similar question for its moduli space, i.e.,
its quotient by  
the full diffeomorphism group of $M$, acting by pulling back metrics.

These spaces are customarily equipped with the topology of smooth convergence on compact subsets and the 
 quotient topology, respectively (and precise definitions of all of these concepts will be given in the following section).
The topological properties of these objects hence provide the right means to measure "how many"
different metrics and geometries the manifold $M$ does exhibit,
and since Weyl's early result on the connectedness of the space
of positive Gaussian curvature metrics on $S^2$ (\cite{Weyl}) 
and the foundings of Teichm\"uller, infinte-dimensional manifold and Lie group theory,
uniformization and geometrization, the study of spaces of metrics and their moduli has been a topic of  interest for
differential geometers, global and geometric analysts and topologists alike.

Especially in recent years there has been intensive activity and substantial further progress on these issues, compare, for example,
\cite{bamler_kleiner},
\cite{banakhar:_spaces},
\cite{belegradekar:_gromov_hausd},
\cite{Bel18},
\cite{MR3619306},
\cite{MR2863912},
\cite{BKS2},
\cite{belegradek15:_connec},
\cite{botvinnik-ebert-randalwilliams},
\cite{botvinnik-ebert-wraith},
\cite{botvinnik95},
\cite{botvinnik96:_metric},
\cite{botvinnik10:_homot},
\cite{botvinnik17:_homot_ricci},
\cite{MR3551844},
\cite{Bu},
\cite{carr88:_const},
\cite{chen:_path},
\cite{chernysh:_r_m},
\cite{coda12:_defor},
\cite{crowley-schick},
\cite{crowley-schick-steimle},
\cite{D17},
\cite{DG19},
\cite{dessai16:_noncon_modul_spaces_nonneg_section},
\cite{ebert-randal-williams1},
\cite{ebert-randal-williams2},
\cite{ebert-wiemeler},
\cite{MR3660235},
\cite{MR3451392},
\cite{farrell09:_teich},
\cite{farrell10},
\cite{farrell10_2},
\cite{MR3629483},
\cite{Frenck1},
\cite{Frenck-Reinhold},
\cite{gajer87:_rieman},
\cite{Goo17},
\cite{gromov83:_posit_dirac_rieman},
\cite{hanke14:_space},
\cite{hitchin74:_harmon_spinor},
\cite{MR2198806},
\cite{KKRW20},
\cite{kreck93:_noncon},
\cite{lawson89:_spin},
\cite{lohkamp95:_curvat},
\cite{rosenberg01:_metric},
\cite{ruberman01:_posit_seiber_witten},
\cite{wraith16:_non},
\cite{tuschmann15:_modul_spaces_rieman_metric},
\cite{tuschmann_survey},
\cite{walsh11:_metric_morse_part_i},
\cite{walsh13:_cobor},
\cite{walsh14:_h},
\cite{wiemeler2},
\cite{wiemeler},
\cite{wraith11:_ricci},
\cite{wraith16:_path},
\cite{wraith_Oberwolfach}.

\

Many of the works just cited either address (spaces and/or) moduli spaces 
of positive scalar curvature metrics as well as corresponding modified
moduli spaces like, for example,  the so-called observer moduli space,
or else are concerned with spaces or moduli spaces of metrics with negative sectional curvature. 

\

On the other hand, results on non-negative Ricci or non-negative sectional curvature metrics have
remained quite scarce so far. 

\

Indeed, for closed manifolds and the (genuine) moduli spaces of metrics of these types,  
all results in general dimension that are known so far
(compare \cite{dessai16:_noncon_modul_spaces_nonneg_section},  \cite{DG19})
only show that there are manifolds for which 
the moduli spaces of metrics with non-negative sectional curvature 
are not connected and can even have an infinite number of 
components. 
Moreover, in \cite{wraith16:_non}
an analogous result is shown for spaces 
of non-negative Ricci curvature
metrics. 
Also, there are results about 
topological properties of the space of non-negatively curved metrics on two-spheres
and real projective planes and on the Gromov-Hausdorff metric on the set of isometry classes 
of non-negatively curved two-spheres, compare \cite{banakhar:_spaces},
\cite{belegradekar:_gromov_hausd}.
(For the results which are known in the non-compact case please see the paragraph following Remark \ref{sec:introduction-5} below.)

\ 

In particular, to the best of the authors' knowledge, the following theorem is the very first result
about the rational homotopy groups and the rational cohomology 
 of the moduli space of Ricci non-negative metrics on closed manifolds:

\

\begin{theorem}
\label{sec:products-with-higher-3-intro}
  Let \(M\) be a simply connected closed smooth manifold which admits a metric
  with non-negative Ricci curvature and let  \(T\) be a torus of dimension \(n\geq 4\), \(n\neq 8,9,10\).
  Then the moduli space \(\mathcal{M}_{\Ric\geq 0}(M\times T)\) of non-negatively Ricci curved metrics on $M\times T$
  has non-trivial higher rational cohomology groups and non-trivial higher rational homotopy groups.
\end{theorem}

\

We can even say more than is stated in the above theorem: If \(n=4\) it is the third rational cohomology group (rational homotopy group, respectively)  of the moduli space which is always non-trivial. In the case where \(n>4\) it is the fifth rational cohomology group (rational homotopy group, respectively)  which is non-trivial.

Hence, by taking products of spheres with tori we obtain in particular:

\

\begin{cor}
\label{sec:introduction-3}
  In every dimension \(n\geq 6\) there exist closed smooth manifolds \(M^n\) 
  with \(\pi_1(M)\cong\mathbb{Z}^4\) 
  for which the third rational cohomology group   \(H^3(\mathcal{M}_{\Ric\geq 0}(M);\mathbb{Q})\) and the third rational homotopy group \(\pi_3(\mathcal{M}_{\Ric\geq 0}(M))\otimes \mathbb{Q}\)
  of the moduli space of non-negatively Ricci curved metrics are non-trivial.
\end{cor}

\

Notice that in the above corollary the rank of \(\pi_1(M^n)\) is constant and, in particular, independent of the dimension $n$.
If one does not want to impose a bound on the rank of the fundamental group, one could also take \(M\) to be a high-dimensional torus and arrive at a similar conclusion as in the corollary
(see below).
 We think that it is an interesting question if this bound on the rank can be further improved. In particular, we would like to ask 
 if there are simply connected manifolds for which the conclusion of the above corollary also holds.

\

Since products of Ricci flat K3 surfaces and flat tori are Ricci flat (but not flat  nor do admit positive scalar curvature), another new consequence of Theorem \ref{sec:products-with-higher-3-intro} is that 
{\em in each dimension $n \geq 8$ there exist closed smooth manifolds $M$ for which the moduli space of Ricci flat metrics has non-trivial higher rational homotopy and cohomology groups.}

\

Let us now turn to other consequences and applications, namely, estimates for the number of path components of moduli spaces
of metrics with non-negative Ricci or non-negative sectional curvature. 
Notice here  in particular that a lower bound on the number of components in a moduli space is always 
also a lower bound on the number of components for the respective space of metrics.

\

\begin{theorem}
\label{sec:introduction-6}
  Let \(M\) be a simply connected closed smooth manifold
  which admits a Riemannian metric with non-negative Ricci curvature, and let $T^k$ denote
  a $k$-dimensional standard smooth torus, where $k\in\N$. 
  
  Then the moduli space \(\mathcal{M}_{\Ric\geq 0}(M\times T^k)\)  of Riemannian metrics 
  with non-negative Ricci curvature on $M\times T^k$   
 has at least as many path components as the moduli space  \(\mathcal{M}_{\Ric \geq 0}(M)\)
 of metrics with non-negative Ricci curvature on $M$.
 
 Moreover, this statement carries over  to all manifolds and moduli spaces of Riemannian metrics 
 for which the Cheeger-Gromoll Splitting Theorem holds, e.g., for metrics of non-negative sectional curvature as well as Ricci flat metrics.
 \end{theorem}

\

By \cite{dessai16:_noncon_modul_spaces_nonneg_section}, in each dimension \(4k+3\geq 7\) there exist infinite sequences of closed manifolds \(M\) such that the moduli space \(\mathcal{M}_{\sec\geq 0}(M)\) of Riemannian metrics with non-negative sectional curvature on \(M\) has infinitely many path components.
Taking products of these examples with tori and applying Theorem~\ref{sec:introduction-6}, we obtain:

\

\begin{cor}
  \label{sec:introduction-2}In every dimension $n\geq 7$ there exist infinite sequences of closed smooth manifolds $M$ of pairwise distinct homotopy type for which the moduli space $\mathcal{M}_{\sec \geq 0}(M)$ of Riemannian metrics with non-negative sectional curvature has infinitely many path components.
\end{cor}

\

\begin{remark}  We note that such examples have so far been known only in dimensions $4k+3$ where $k\geq 1$, compare \cite{dessai16:_noncon_modul_spaces_nonneg_section}.
Furthermore, by taking products of the manifolds from the above corollary with the real line and applying Theorem~\ref{sec:introduction-4} below, we can extend results of \cite{dessai16:_noncon_modul_spaces_nonneg_section} on non-compact
spaces which were only known for manifolds of dimension $8n$, $n\geq 1$ and obtain:
\end{remark}

\

\begin{cor}
  \label{sec:introduction-7}
In every dimension $n\geq 8$ there exist infinite sequences of non-compact connected smooth manifolds $M$ of pairwise distinct homotopy type for which the moduli space $\mathcal{M}_{\sec \geq 0}(M)$ of complete 
Riemannian metrics with non-negative sectional curvature has infinitely many path components.
\end{cor}

\

\begin{remark}
\label{sec:introduction-5} The first examples of closed (respectively, open) smooth manifolds whose \emph{space} 
 of metrics with non-negative
 Ricci curvature
has infinitely many path-components were constructed only recently by Schick and Wraith \cite{wraith16:_non}.
 Note that by Theorem 1.8 in that paper, the moduli space \(\mathcal{M}_{\Ric\geq 0}(M)\), where \(M\) is one of the examples constructed in \cite{dessai16:_noncon_modul_spaces_nonneg_section} has infinitely many components. Therefore we can state our Corollaries \ref{sec:introduction-2} and~\ref{sec:introduction-7} verbatim also for moduli spaces of Riemannian metrics
 with non-negative Ricci curvature.
\end{remark} 

\

Staying with non-compact manifolds, we would also  like to point out that
the soul theorem of Cheeger and Gromoll \cite{MR0309010} 
has been used by Kapovitch--Petrunin--Tuschmann \cite{MR2198806}, 
Belegradek--Kwasik--Schultz \cite{MR2863912}, 
\cite{BKS2} and
Belegradek--Farrell--Ka\-po\-vitch \cite{MR3619306} to study spaces and moduli spaces
 of complete non-negative sectional curvature metrics on open manifolds.
 Both  \cite{MR2198806} and \cite{MR2863912}  obtain examples of manifolds for which
the moduli spaces 
of non-negative sectional curvature metrics have an infinite number of components,
but in each dimension produce only finitely many such manifolds.
In  \cite{MR3619306} it is shown that
the  space  (not moduli space) of non-negative sectional curvature metrics 
can have  higher homotopy groups which can be non-trivial and contain elements of infinite order.
In \cite{Bel18}
it is shown that  the space of complete Riemannian metrics of nonnegative sectional curvature on certain open spin manifolds has nontrivial homotopy groups with elements of order two in infinitely many degrees.

These are the only results on (moduli) spaces on non-negative sectional curvature metrics 
on open manifolds of dimension greater than two known so far.
For the space of such metrics on the two-plane there are also results about connectedness properties
by  Belegradek and Hu \cite{belegradek15:_connec}.

\

In conjunction with Corollary \ref{sec:introduction-3}, the following theorem therefore  also implies the first
non-triviality results for higher rational cohomology and homotopy groups of moduli spaces of non-negatively sectional or Ricci-curved metrics on certain open manifolds:

\

\begin{theorem}
\label{sec:introduction-4}
  Let \(M\) be a closed manifold with non-negative Ricci curvature. Then
\(\mathcal{M}_{\Ric\geq 0}(M)\)
is homeomorphic to an open and closed subset of
     \( {\mathcal{M}}_{\Ric\geq 0}(M\times \mathbb{R})\).
In particular, \( {\mathcal{M}}_{\Ric\geq 0}(M\times \mathbb{R})\) has at least as many path components (and at least as complicated homotopy and cohomology)  as \(\mathcal{M}_{\Ric\geq 0}(M)\).
\end{theorem}

(For a more precise version of Theorem \ref{sec:introduction-4} see Corollary~\ref{sec:metr-prod-with}.)

\

\

To obtain the results of the present paper, we employ 
 in a new way the splitting theorem of Cheeger and Gromoll 
\cite{MR0303460}. It allows us to construct maps from moduli spaces to other spaces whose topology is easier to understand.

For example, by the splitting theorem, a closed non-negatively Ricci-curved manifold \((M,g)\) with \(\pi_1(M)=\mathbb{Z}^n\) is isometric to a bundle with simply connected non-negative Ricci curved fiber \((N,h)\) over a flat \(n\)-dimensional torus \((T^n,h')\). 
We define maps
\begin{align*}
  \mathcal{M}_{\Ric \geq 0}(M)&\rightarrow \mathcal{M}_{\Ric\geq 0}(N)&\mathcal{M}_{\Ric \geq 0}(M)&\rightarrow \mathcal{M}_{\sec = 0}(T^n)\\
  [g]&\mapsto[h]&[g]&\mapsto [h'].
\end{align*}
If \(M\) is diffeomorphic to the product \(N\times T^n\), a crucial step in our arguments is then to prove that these maps 
are retractions. 
Therefore the topology of \(\mathcal{M}_{\Ric\geq 0}(M)\) is at least as complicated as the topology of \(\mathcal{M}_{\Ric\geq 0}(N)\) and \(\mathcal{M}_{\sec = 0}(T^n)\).

It is known that for certain choices of \((4k+3)\)-dimensional manifolds \(N\), where \(k\geq 1\), the moduli space \(\mathcal{M}_{\Ric\geq 0}(N)\) is non-connected.
Hence, by the above argument, we obtain non-connectedness results for moduli spaces in arbitrary dimensions \(4k+3+n\geq 7\) with \(n=0,1,2,\ldots\in \mathbb{N}\).

Another  ingredient in the proofs of the above results consists in having knowledge
about the moduli space of flat metrics on an $n$-dimensional torus.
By classical work of Wolf \cite{MR0322748}, it is known that the moduli space
\(\mathcal{M}_{\sec = 0}(T^n)\) of such metrics  is homeomorphic to the biquotient space \(O(n)\backslash GL(n,\R)/GL(n;\mathbb{Z})\).
However, and quite surprisingly, it seems that until now the topology of this moduli space 
has never been further studied by geometers (compare here, e.g., also the corresponding remarks in \cite{BDP}).

But let us now observe that this latter space is actually a model for the classifying space \(B_{\mathcal{FIN}}GL(n,\mathbb{Z})\) for the family of finite subgroups of \(GL(n,\mathbb{Z})\). This fact does not seem to have been noticed before in the context of moduli spaces.
But it was, of course, observed in the context of classifying spaces (see for example \cite{MR2195456}) which indeed have been studied intensively by topologists.
Employing their work (see Proposition~\ref{sec:path-comp-rati} for details) 
in our study of moduli spaces, we can then detect non-trivial classes in rational cohomology groups and rational homotopy groups of \(\mathcal{M}_{\Ric\geq 0}(M)\).

\

\begin{remark}
A natural question to ask is if the main results of this paper do also apply when the products with tori used here 
will get replaced by products with some general
closed flat manifold. Thus, let \(M\) be diffeomorphic to \(N\times F\), where \(N\) is a closed simply connected manifold with non-negative Ricci curvature and \(F\) is a closed flat manifold.

   Then, using arguments similar to the ones outlined above and detailed in the following sections, one can define corresponding retractions
\begin{align*}
  \mathcal{M}'_{\Ric \geq 0}(M)&\rightarrow \mathcal{M}_{\Ric\geq 0}(N)& \text{and} \quad \qquad \mathcal{M}_{\Ric \geq 0}(M)&\rightarrow \mathcal{M}_{\sec = 0}(F)
\end{align*}
from a suitable  open and closed subset \(\mathcal{M}'_{\Ric \geq 0}(M)\) of \(\mathcal{M}_{\Ric \geq 0}(M)\) onto $\mathcal{M}_{\Ric\geq 0}(N)$ 
and from  \(\mathcal{M}_{\Ric \geq 0}(M)\) onto $\mathcal{M}_{\sec = 0}(F)$, respectively. 

Therefore Theorem~\ref{sec:introduction-6} does also hold when the $k$-torus $T^k$
figuring there is replaced by any other closed flat manifold $F$. Moreover,
Theorem~\ref{sec:products-with-higher-3-intro} and its conclusions will carry over as well, provided that
the moduli space of flat metrics on \(F\) does exhibit corresponding topological properties as the moduli space  of flat metrics on $T^k$.
In particular, it is likely that, with respect to the number of dimensions that
they are covering in their present form,  both Theorem 1.1 and Corollary 1.2 can still be slightly improved. \end{remark}

\

The remaining parts of this article are organised as follows:
In section 2 we introduce relevant notations and state some preliminaries. 
We discuss non-negatively Ricci curved metrics on products \(M\times \mathbb{R}\) in section \ref{sec:metrics-mtimes-r-1}.
In section \ref{sec:metrics-mtimes-s1-3}  we investigate metrics on \(M\times S^1\).
Section \ref{sec:products-with-higher}  is devoted to a study of products of simply connected manifold \(M\)  with higher dimensional tori
and the components and cohomology groups of their moduli spaces.

\

It is our pleasure to thank Igor Belegradek, Boris Botvinnik, Anand Dessai, Stephan Klaus, and David Wraith for their comments on an earlier version of this paper,
as well as the anonymous referees for providing further remarks and their careful reading.

\

\section{Spaces and moduli spaces of metrics and Ricci souls}

\subsection{Spaces and moduli spaces of Riemannian metrics}

In this subsection we recall some properties of spaces and moduli spaces of Riemannian metrics. Let us start with introducing some notation.

\begin{definition}
  Let \(M\) be a smooth manifold. Then we denote by
  \begin{enumerate}
  \item \(\mathcal{R}^r(M)\) the space of complete smooth (i.e. \(C^\infty\)) Riemannian metrics on \(M\) equipped with the \(C^r\)-topology of uniform convergence on compact subsets of \(M\).
  \item \(\mathcal{M}^r(M)=\mathcal{R}^r(M)/\Diff(M)\) the moduli space of metrics on \(M\) equipped with the quotient topology.
    Here the diffeomorphism group \(\Diff(M)\) of \(M\) acts by pulling back metrics.
  \item \(\mathcal{R}^r(M\times \mathbb{R},M)\) the space of metrics on \(M\times \mathbb{R}\) which are isometric to product metrics \(h+dt^2\), where \(h\) is a complete Riemannian metric on \(M\), equipped with the \(C^r\)-topology of uniform convergence on compact subsets of \(M\times \mathbb{R}\).
       \item \(\mathcal{M}^r(M\times \mathbb{R},M)=\mathcal{R}^r(M\times \mathbb{R},M)/\Diff(M\times \mathbb{R})\) equipped with the quotient topology.
  \end{enumerate}
  We will add subscripts to the above notation to indicate curvature constraints.
\end{definition}

Note that the space \(\mathcal{R}^r(M)\) is convex.
Therefore all of its open subsets (and also all open subsets of the moduli space \(\mathcal{M}^r(M)\)) are locally path connected.
Hence their connected components coincide with their path components.

Notice also that since the natural map from a given space of metrics to its corresponding moduli space is surjective, the number of components of the moduli space is a lower bound for the number of components of the space itself.

We will show in Theorem \ref{sec:metrics-mtimes-r} that \(\mathcal{R}_{\Ric\geq 0}^r(M\times \mathbb{R},M)\) (and \(\mathcal{M}_{\Ric\geq 0}^r(M\times \mathbb{R},M)\)) is an open and closed subspace of \(\mathcal{R}_{\Ric\geq 0}^r(M\times \mathbb{R})\) (and \(\mathcal{M}_{\Ric\geq 0}^r(M\times \mathbb{R})\), respectively).
So lower bounds for the number of components of the former space (and moduli space) give lower bounds for the number of components of the latter space (and moduli space, respectively).

\subsection{Souls and Ricci souls}

Recall that
the Cheeger-Gromoll Soul Theorem \cite{MR0309010} states that
a complete open Riemannian manifold \((M,g)\) of non-negative sectional curvature is diffeomorphic to the normal bundle of a closed convex totally geodesic submanifold \(S\subset M\), the so-called soul of \(M\).

Moreover, the Cheeger-Gromoll Splitting Theorem \cite{MR0303460} says that 
if \((M,g)\) is a complete manifold of non-negative Ricci curvature which contains a line, then
 \((M,g)\) is isometric to a Riemannian product \((N,h)\times \mathbb{R}^k\), where \((N,h)\) is a complete manifold of non-negative Ricci curvature which does not contain a line, and where  \(\R^k\) is flat \(k\)-dimensional Euclidean space.

 Notice that if \((M,g)\) is the Riemannian universal  covering of a closed non-negatively Ricci curved manifold, then the manifold \(N\) given by the Splitting Theorem is closed.

 Motivated by the Soul Theorem we introduce  for non-negatively Ricci curved manifolds the following notation.

\begin{definition}
  \begin{enumerate}
  \item   If \((M,g)\) is a simply connected non-negatively Ricci curved complete Riemannian manifold that is isometric to a Riemannian product \((N,h)\times \mathbb{R}^k\) with \(N\) closed, we call \((N,h)\) the \emph{Ricci soul} of \((M,g)\).
  \item If \((M,g)\) is a closed Riemannian manifold of non-negative Ricci curvature, 
 we call the Ricci soul \((N,h)\) of the  Riemannian universal covering of \((M,g)\) the \emph{Ricci soul} of \((M,g)\).
  \end{enumerate}
\end{definition}

\subsection{Geodesics and lines}

In this section we collect some 
facts about geodesics and lines in Riemannian manifolds that we will put to use in later sections of this work.

\begin{lemma}
  \label{sec:souls-ricci-souls}
  Let \((M,g)=(N,h)\times \mathbb{R}\) with \(N\) compact and \(x\in N\).
  Then there is a unique unit speed line \(\gamma_x:\mathbb{R}\rightarrow M\) passing from \(-\infty\) to \(+\infty\) with \(\gamma_x(0)=(x,0)\).
  Indeed, \(\gamma_x\) is given by \(\gamma_x(t)=(x,t)\).
\end{lemma}
\begin{proof}
  Let \(\gamma:\mathbb{R}\rightarrow M\) be a unit speed line such that
  \begin{align*}
    \gamma(t)&=(\gamma_N(t),\gamma_{\mathbb{R}}(t))&\gamma(0)&=(x,0)&\lim_{t\to\pm\infty} \gamma_{\mathbb{R}}(t)&=\pm \infty.
  \end{align*}
  Then \(\gamma_N\) and \(\gamma_{\mathbb{R}}\) are geodesics in \(N\) and \(\mathbb{R}\), respectively. In particular, \(\gamma_{\mathbb{R}}\) is of the form
  \(\gamma_{\mathbb{R}}(t)=ct,\)
  with \(0\leq c\leq 1\). We then get from the triangle inequality
  \begin{align*}
    |t|&=d(\gamma(0),\gamma(t))\\
    &\leq d((x,0),(\gamma_N(t),0))+d((\gamma_N(t),0),(\gamma_N(t),\gamma_{\mathbb{R}}(t)))\\
                   &\leq \diam(N,h) + c|t|.
  \end{align*}
  Since \(N\) is compact, for \(t\to \pm \infty\), this inequality can only hold if \(c=1\).
  Because \(\gamma\) has unit speed it follows that \(\gamma=\gamma_x\) and the claim follows.
\end{proof}

\begin{remark}
  In the same way one can prove  a generalised version of Lemma \ref{sec:souls-ricci-souls} 
  for products of the form \((N,h)\times \mathbb{R}^k\), $k\ge 1$, with \(N\) compact. In these products all unit speed lines 
  are then of the form \(\gamma(t)=(x,v_1t+v_2)\) with \(x\in N\), \(v_1,v_2\in \mathbb{R}^k\) and \(\|v_1\|=1\).
\end{remark}

\begin{lemma}
  \label{sec:souls-ricci-souls-1}
  Let \(M\) be a smooth manifold and \((g_n)\) be a sequence of Riemannian metrics on \(M\) which converges in the \(C^1\)-topology of uniform convergence on compact subsets of \(M\) to a metric \(g_0\).
  Moreover, let \(\gamma_n:TM\times \mathbb{R}\rightarrow M\) be mappings such that, for \(v\in TM\), \(\gamma_n(v,\cdot)\) is the geodesic with initial values \(\gamma_n(v,0)=\pi(v)\), and such that \(\gamma_n'(v,0)=\frac{\partial \gamma_n}{\partial t}(v,0)=v\)   with respect to the metric \(g_n\),
  where \(\pi:TM\rightarrow M\) is the canonical  projection.

  Then \((\gamma_n,\gamma_n')\) converges in the \(C^0\)-topology of uniform convergence on compact subsets of \(TM\times \mathbb{R}\) to \((\gamma_0,\gamma_0')\).
  In particular, \(\gamma_n(v,\cdot)\) converges in the \(C^1\)-topology of uniform convergence on compact subsets of \(\mathbb{R}\) to \(\gamma_0(v,\cdot)\).
\end{lemma}
\begin{proof}
  This fact can easily be deduced from \cite{sakai83}, but for the convenience of the reader 
let us include a short own proof here. Indeed, in local coordinates we may write
  \begin{align*}
    \gamma_n&=(\gamma_n^i)_{i=1,\dots,m}&\gamma_n'=(\xi_n^i)_{i=1,\dots,m}.
  \end{align*}

  Then \((\gamma_n,\gamma'_n)\) is a solution to the system of ODEs
  \begin{align*}
    \frac{\partial \gamma_n^i}{\partial t}&=\xi_n^i &i&=1,\dots,m,\\
    \frac{\partial \xi_n^i}{\partial t}&=-\sum_{j,k=1}^m \Gamma^i_{njk}\xi_n^j\xi_n^k&i&=1,\dots,m,
  \end{align*}
  with initial value \((\gamma_n(v,0),\gamma_n'(v,0))=(\pi(v),v)\).
  Here the \(\Gamma^i_{njk}\) are the Christoffel symbols of the metric \(g_n\).
  Note that the \(\Gamma_{njk}^i\) converge in \(C^0\)-topology to \(\Gamma_{0jk}^i\) because the metrics converge in \(C^1\)-topology.
  
  Since the solution of a system of ODEs depends continuously on the parameters, the claim follows.
\end{proof}

\section{Metrics on products with the Euclidean line $\R$}
\label{sec:metrics-mtimes-r-1}

In this section we prove the following theorem and apply it to moduli spaces.

\begin{theorem}
\label{sec:metrics-mtimes-r}
  Let \((M,g_0)\) be a complete connected Riemannian manifold which is isometric to \((N_0,h_0)\times \R\)  
  where \((N_0,h_0)\) is closed and \(\mathbb{R}\) is equipped with the standard metric. Moreover, let \(g_n\) be  metrics on \(M\) which converge,
  as $n\to\infty$, for  \(r> 1\) to  \(g_0\)  with respect to 
  the \(C^r\)- topology of
  uniform convergence on compact subsets.
   Assume that there exist  isometries \(\psi_n:(M,g_n)\rightarrow (N_n,h_n)\times\mathbb{R}\) where \((N_n,h_n)\) is closed.
Then

\begin{itemize}
  \item for sufficiently large \(n\) there are diffeomorphisms \(\phi_n: N_0 \rightarrow N_n\), such that
  \item the metrics  \(\phi_n^*h_n\) converge in the \(C^r\)- topology to \(h_0\).
  \end{itemize}
\end{theorem}

\begin{proof}
  Notice that \(M\) has two ends and that we can assume that under \(\psi_n\) the first end is mapped to \(-\infty\) and the second end to \(+\infty\).

  For \(x\in M\) let \(\gamma_{x,n}:\mathbb{R}\rightarrow (M,g_n)\) be the unique line from the first end to the second end passing through \(\gamma_{x,n}(0)=x\).

By assumption \(g_n\) converges in the \(C^2\)- topology to \(g_0\). Moreover,  the set \(\{v\in T_xM;\; \frac{1}{2}\leq g_0(v,v)\leq 2\}\) is compact and contains, for sufficiently large \(n\), the unit sphere in \(T_xM\) with respect to the metric \(g_n\). Therefore, by Lemma~\ref{sec:souls-ricci-souls-1}, for every subsequence \(\gamma_{x,n_l}\), there is a subsequence \(\gamma_{x,n_{l_j}}\) such that the \(\gamma_{x,n_{l_j}}\) converge in the \(C^1\)- topology of uniform convergence
 on compact subsets of \(\mathbb{R}\) to a geodesic in \((M,g_0)\). This geodesic stretches from the first end to the second end and passes through \(x\). We claim that this geodesic \(\gamma\)  is a line in \((M,g_0)\) and hence is equal to \(\gamma_{x,0}\).
Indeed, if \(t_1,t_2\in \R\), \(t_1<t_2\), then we have for all \(n\):
\begin{equation*}
  d_n(\gamma_{x,n_{l_j}}(t_1),\gamma_{x,n_{l_j}}(t_2))=t_2-t_1.
\end{equation*}
Since \(g_n\) converges uniformly on compact subsets to \(g_0\) and \(\gamma_{x,n_{l_j}}\) converges uniformly on compact sets to \(\gamma\), we have that the left hand side of the above equation converges to \(d_0(\gamma(t_1),\gamma(t_2))\).
Because the right hand side is constant, it follows that \(\gamma\) is a line and hence equal to \(\gamma_{x,0}\).

The above argument holds for \emph{all} subsequences \(\gamma_{x,n_l}\). Therefore we have that  \(\gamma_{x,n}\rightarrow \gamma_{x,0}\) in \(C^1\)- topology uniformly on compact subsets.
Note, moreover, that the lines \(\gamma_{x,n}\) depend continuously on \(x\in M\).
Furthermore, the above convergence is also uniform on compact subsets of \(M\).

Now let \(X_n\) be the vector field on \(M\) given by \(X_n(x):=\frac{d\gamma_{x,n}}{dt}(0)\). Since the integral curve of \(X_n\) passing through \(x\) is given by \(\gamma_{x,n}\), \(X_n\) then converges in \(C^0\)-topology to \(X_0\) uniformly on compact sets.

Let us now identify each \(N_n\) with the image under the inclusion map

\begin{equation*}
N_n=N_{n}\times \{0\}\hookrightarrow N_n\times \mathbb{R}\rightarrow_{\psi_n^{-1}}M.  
\end{equation*}

This identifies \(T_pN_n\) with \( \langle X_n(\psi_n^{-1}(p,0))\rangle^{\perp_{g_n}} \subset T_{\psi_n^{-1}(p,0)} M\).

Moreover, the differential of the projection
\begin{equation*}
  \psi_n':M\rightarrow_{\psi_n}N_n\times \mathbb{R}\rightarrow N_n
\end{equation*}
is then just the orthogonal projection
\[d\psi_n'(v)=v-g_n(v,X_n)X_n.\]

This means that, for sufficiently large \(n\), \(\phi_n=\psi_n'\circ \psi_0^{-1}\circ \iota_0\) is an immersion and therefore a covering.
Here \(\iota_0:N_0\times\{0\}\hookrightarrow N_0\times \mathbb{R}\) denotes the inclusion.

Since \(\phi_n\) induces an isomorphism on fundamental groups, it follows that \(\phi_n\) is a diffeomorphism.

Therefore if \(X,Y,Z\) are vector fields on \(N_0\), we can compute
\begin{align*}
  X(\phi_n^*h_n(Y,Z))&=Xg_n(Y,Z)-X(g_n(Y,X_n)g_n(Z,X_n))\\
  &= Xg_n(Y,Z)-(g_n(\nabla^n_X Y,X_n)+g_n(Y,\nabla^n_X X_n))g_n(Z,X_n)\\ &-(g_n(\nabla^n_X Z,X_n)+g_n(Z,\nabla^n_X X_n))g_n(Y,X_n)\\
 &=Xg_n(Y,Z)-g_n(\nabla^n_X Y,X_n)g_n(Z,X_n)-g_n(\nabla^n_X Z,X_n)g_n(Y,X_n)\\
&\rightarrow Xg_0(Y,Z)-g_0(\nabla^0_X Y,X_0)g_0(Z,X_0)-g_0(\nabla^0_X Z,X_0)g_0(Y,X_0)\\ &=Xg_0(Y,Z)=Xh_0(Y,Z) \text{ as } n\rightarrow \infty.
\end{align*}
Here \(\nabla^n\) denotes the Riemannian connection on \((M,g_n)\).
Note here that under the isometry \((M,g_n)\cong (N_n\times \R,h_n+ dt^2)\), the vector field \(X_n\) corresponds to the constant length vector field \(\partial/\partial t\) on \(N_n\times \R\) tangent to the second factor.
Hence, \(\nabla^n X_n=0\) follows.
Notice, moreover, that the above convergence is uniform on \(N_0\).

 We can deal with higher derivatives of the metrics in a similar way, and hence the claim follows.
\end{proof}

\

\begin{remark}
  In Theorem~\ref{sec:metrics-mtimes-t} below we will generalize Theorem~\ref{sec:metrics-mtimes-r} and its proof to the case of products of \(M\) with \(\mathbb{R}^k\), \(k>1\).
  However, treating the case \(k=1\) separately allows for keeping the notation in both the proofs simpler. Moreover, we think it might be also helpful for the reader to see the arguments in the simplest case first.
\end{remark}

\begin{cor}
\label{sec:metr-prod-with}
  Let \(M\) be a closed manifold. Then for \(r> 1\) there is a homeomorphism
  \begin{equation*}
    {\mathcal{M}}_{\Ric\geq 0}^{r}(M\times \mathbb{R},M)\rightarrow \mathcal{M}_{\Ric\geq 0}^{r}(M).
  \end{equation*}
\end{cor}
\begin{proof}
  The inverse map of the above map is induced by the map which assigns to a metric on \(M\) its metrical product with a line.

Therefore it suffices to construct a continuous \(\Diff(M\times \R)\)-invariant map
 \begin{equation*}
    {\mathcal{R}}_{\Ric\geq 0}^r(M\times \mathbb{R},M)\rightarrow \mathcal{M}_{\Ric\geq 0}^r(M).
  \end{equation*}
We define this map to be the map which assigns to a metric on \(M\times \R\) of non-negative Ricci-curvature
 the restriction of that metric to an integral submanifold \(M_{\mathcal{D}}\)  of the distribution \(\mathcal{D}\)  on \(M\times \R\) which is orthogonal to the lines with respect to this metric.

 In a Riemannian product \((M,h)\times \mathbb{R}\), where \(M\) is closed, all lines are of the form \(\gamma_x(t)=(x,t)\) for \(x\in M\). Hence the integral submanifold \(M_{\mathcal{D}}\) exists and is diffeomorphic to \(M\).
 Moreover, the metric on the integral submanifold \(M_{\mathcal{D}}\)  is non-negatively Ricci-/sectional curved if and only if the metric on the product is curved in this way.

Furthermore,
by Theorem \ref{sec:metrics-mtimes-r} above this assignment is continuous.
Hence the claim follows.
\end{proof}

\

\begin{remark}
  In the special case of non-negative \emph{sectional} curvature, the proof of Corollary~\ref{sec:metr-prod-with} also yields a homeomorphism
    \begin{equation*}
    {\mathcal{M}}_{\sec\geq 0}^{r}(M\times \mathbb{R},M)\rightarrow \mathcal{M}_{\sec\geq 0}^{r}(M).
  \end{equation*}
  This homeomorphism is given by sending a metric \(g+dt^2\) on \(M\times \mathbb{R}\) to its soul \((M,g)\).
\end{remark}

\

\begin{remark}
 In the case \(2\leq r<\infty\) our argument above corrects  and fills 
  a gap in the proof of Proposition 2.8 of \cite{MR2863912}. There, based on the Splitting Theorem, it is claimed that 
  the map \(\phi\)  sending a metric \(g\) in \(\mathcal{M}_{\sec\geq 0}^{r_0}(M_i)\) to its product \(g+dt^2\) with a line gives a homeomorphism
\[\coprod_i\mathcal{M}_{\sec\geq 0}^{r_0}(M_i)\rightarrow \mathcal{M}_{\sec\geq 0}^r(M\times \R),\]
where \(r_0=r\) if \(r=0\) or \(r=\infty\) and \(r_0=r-1\) otherwise. Here the union is taken over all manifolds \(M_i\) such that \(M_i\times \mathbb{R}\) and \(M\times \mathbb{R}\) are diffeomorphic.
The map \(\phi\) can be understood as being the inverse of sending a metric on \(M\times\mathbb{R}\) to its soul \((M,g)\).
But \(\phi\) is discontinuous unless \(r=0\) or \(r=\infty\). (See also \cite{belegradek18:_modul}.)
\end{remark}

\section{Metrics on products with the circle $S^1$}
\label{sec:metrics-mtimes-s1-3}

In this section we consider connectedness properties of 
the space and moduli space of non-negative Ricci-curved metrics on \(M\times S^1\),
where \(M\) is a simply connected closed manifold.
The following lemma will turn out to be the key structural result in this study.

\begin{lemma}
\label{sec:metrics-mtimes-s1}
  Let \(M\) be a closed connected Riemannian manifold with torsion free fundamental group
  \(\pi_1(M)\) and non-negative Ricci curvature. Then there is a simply connected closed Riemannian manifold \((N,h)\) with non-negative Ricci curvature and a closed flat manifold \(F\) with fundamental group \(\pi_1(F)\cong \pi_1(M)\) such that \(M\) is isometric to an \((N,h)\)-bundle with structure group \(\pi_1(F)\) associated to the universal covering \(\R^n\rightarrow F\). Here \(\pi_1(F)\) acts via a homomorphism \(\varphi:\pi_1(F)\rightarrow \Iso(N,h)\) on \(N\).
\end{lemma}

\begin{proof}
  By the Splitting Theorem of Cheeger and Gromoll \cite{MR0303460} the universal covering of \(M\) is isometric to \((N,h)\times \R^n\) and \(\pi_1(M)\) is a discrete subgroup of \(\Iso((N,h)\times \R^n)\cong \Iso(N,h)\times \Iso(\R^n)\).
  Since \(\pi_1(M)\) is torsion free and \(\Iso(N,h)\) compact, \(\pi_1(M)\cap \Iso(N,h)=\{1\}\) and the projection \(\pi_1(M)\rightarrow \Iso(\R^n)\) is injective. Moreover, also because \(\pi_1(M)\) is torsion free, the image of this map is a Bieberbach group and therefore acts freely and isometrically on \(\R^n\).
  We let \(F\) denote the quotient \(\R^n/\pi_1(M)\).
  The lemma follows then by letting \(\varphi\) be the projection from \(\pi_1(M)\) to \(\Iso(N,h)\).
\end{proof}

\

\begin{remark}
  The torsion-freeness of \(\pi_1(M)\) is necessary to guarantee that \(F\) is a manifold. 
  If this is not the case, one might write \(M\) as an orbifold-bundle over a flat orbifold \(\mathcal{F}\) with fiber a manifold \(N\).
  But in this case the fundamental groups of \(\mathcal{F}\) and \(N\) as well as the structure group might not be the same as in Lemma~\ref{sec:metrics-mtimes-s1}. (Compare here also \cite{belegradek04:_nonneg})
\end{remark}

\

\begin{remark}
  By the Bieberbach theorems, the manifold \(F\) in the above lemma is determined up to 
  affine diffeomorphism by \(\pi_1(M)\).
\end{remark}

\

\begin{theorem}
\label{sec:metrics-mtimes-s1-1}
  Let \((M,g_0)\) be a closed connected Riemannian manifold with non-negative Ricci curvature 
  and fundamental group  \(\pi_1(M)\cong\mathbb{Z}\). Moreover, let \(g_n\) be a sequence of metrics of non-negative Ricci curvature on \(M\) which converges in the \(C^\infty\)- topology to the metric \(g_0\). Then there is a sequence of simply connected Riemannian manifolds \((N_n,h_n)\) and \(a_n>0\) for \(n\geq 0\) such that for sufficiently large \(n\) the following holds:
  \begin{itemize}
  \item there are isometries \(\psi_n:(M,g_n)\rightarrow (N_n,h_n)\times_{a_n\mathbb{Z}}\mathbb{R}\) as in the above lemma, which induce the identity on \(\mathbb{Z}=\pi_1(M)\).
  \item there are diffeomorphisms \(\phi_n: N_0 \rightarrow N_n\), such that
  \item the \(\phi_n^*h_n\) converge in the \(C^\infty\)- topology to \(h_0\).
  \item the isometry groups \(\Iso(N_0,\phi_n^*h_n)\) are, for  \(n\) 
  sufficiently large, conjugate in \(\Diff(N_0)\) to subgroups of \(\Iso(N_0,h_0)\) .
  \item the homomorphisms \(\varphi_n:\mathbb{Z}\cong a_n\mathbb{Z}\rightarrow \Iso(N_0,\phi_n^*h_n)\subset \Diff(N_0)\) converge in \(\Diff(N_0)/\Iso(N_0,h_0)\) to \(\varphi_0\). Here \(\Iso(N_0,h_0)\) acts on \(\Diff(N_0)\) by conjugation.
  \end{itemize}
\end{theorem}

\begin{proof}
The first claim follows from Lemma~\ref{sec:metrics-mtimes-s1} above.
The second and third claim follow from Theorem \ref{sec:metrics-mtimes-r} by considering the pullback metrics on the universal covering of \(M\).

The fourth claim then follows from the previous claims by Ebin's slice theorem \cite{MR0267604}, which shows that every metric \(g\)  on \(N_0\) has a neighborhood \(V\) in the space of metrics on \(N_0\) such that the isometry group of 
every \(g'\in V\) is conjugate to a subgroup of the isometry group of \(g\).

To see the last claim we argue as follows:
Let \(\tau \in \pi_1(M)\subset  \Iso((N_0,\phi^*_nh_n)\times \mathbb{R})\subset \Diff(N_0\times \R)\) be a generator of \(\pi_1(M)\).
Using the natural isomorphism \(\Iso((N_0,\phi^*_nh_n)\times \mathbb{R})\cong\Iso(N_0,\phi^*_nh_n)\times \Iso (\mathbb{R})\), we may write \(\tau=(\varphi_n(1),l_n)\).
Then \(l_n\) is the translation by \(\pm a_n\), or in other words \(l_n=f_{\pm a_n}\) where \(f_t\) is the flow of the vector field \(\frac{\partial}{\partial t}\) on \(\R\).

Therefore we have a commutative diagram of the form
\begin{equation*}
\xymatrix{
  (N_0,\phi_n^*h_n)\ar[d]_{\varphi_n(1)}\ar[r]&(N_0\times \R,\tilde{g}_n)\ar[d]^{\tau\circ l_n^{-1}}\\
  (N_0,\phi_n^*h_n)\ar[r]&(N_0\times \R,\tilde{g}_n)}
\end{equation*}

Here \(l_n\) is given by following the flow of the parallel norm-one vector field \(X_n\) which is tangent to the lines in \((N_0\times \R,\tilde{g}_n)\) for time \(a_n\) (compare the proof of \ref{sec:metrics-mtimes-r}). Moreover, 
\(\tau\) is given by the action of a generator of the fundamental group of \(M\) on the universal covering \(N_0\times \R\).

The horizontal maps are given by \(\tilde{\psi}^{-1}_n\circ (\phi_n\times\{0\})\), where the \(\tilde{\psi}_n\) are lifts of the maps \(\psi_n\) to universal coverings.
They are clearly isometric embeddings, and therefore every subsequence of these maps has a subsequence which converges  to an isometric embedding of \((N_0,h_0)\) in \((N_0\times \R,\tilde{g}_0)\) by the Arzela-Ascoli lemma.
This embedding restricts to an isometry of \((N_0,h_0)\).

Since the vector fields \(X_n\) converge to the field \(X_0\), it thus suffices to show that
the sequence of numbers  \(a_n\) converges to \(a_0\).

By shrinking the fibers of \(\frac{1}{m}N_n\rightarrow (M,g_n)\rightarrow F_n=\mathbb{R}/a_n\mathbb{Z}\), one obtains
 a sequence of manifolds which Gromov-Hausdorff converge to \(F_n\) for \(m\rightarrow \infty\).
But if we fix \(m\) and consider the limit \(n\rightarrow \infty\), these manifolds converge to \(\frac{1}{m}N_0\rightarrow (M,g_0)\rightarrow F_0=\mathbb{R}/a_0\mathbb{Z}\).
To see this, note that if \(Y,Z\) are tangent vectors to \(M\) then the shrunk metric \(g_{n,m}\) is given by
\begin{equation*}
  g_{n,m}(Y,Z)=\frac{1}{m}(g_n(Y,Z)-g_n(Y,X_n)g_n(Z,X_n)) + g_n(Y,X_n)g_n(Z,X_n).
\end{equation*}
Hence it follows that if \(g_n\) and \(g_0\) are \(\epsilon\) - close in the \(C^0\)-topology and if 
\(X_n\) is \(\epsilon\) - close to \(X_0\) for some \(\epsilon >0\),
 then \(g_{n,m}\) is \(C\epsilon\) - close to \(g_{0,m}\) for some constant \(C>0\) 
 which is independent of \(\epsilon\) and \(m\).

Thus, by interchanging the order of taking limits, it follows that \(F_n\) Gromov--Hausdorff converges to \(F_0\) as \(n\rightarrow \infty\). 
Whence \(a_n\) converges to \(a_0\) and the claim follows.
\end{proof}

\

\begin{cor}
  \label{sec:metr-prod-with-1}
Let \(M\) be a simply connected closed smooth manifold of dimension at least $5$
 which admits a Riemannian metric with non-negative Ricci curvature.
  Then there is a retraction
  \begin{equation*}
    \mathcal{M}_{Ric\geq 0}(M\times S^1)\rightarrow \mathcal{M}_{\Ric\geq 0}(M).
  \end{equation*} 
  In particular, 
  the moduli space \(\mathcal{M}_{\Ric\geq 0}(M\times S^1)\)  of Riemannian metrics 
  with non-negative Ricci curvature on $M\times S^1$   
 has at least as many path components as the moduli space  \(\mathcal{M}_{\Ric \geq 0}(M)\)
 of metrics with non-negative Ricci curvature on $M$.
\end{cor}

\begin{proof}
  The map
   \begin{equation*}
    \mathcal{M}_{Ric\geq 0}(M\times S^1)\rightarrow \mathcal{M}_{\Ric\geq 0}(M)
  \end{equation*} 
in the above claim is given by first pulling back a metric \(g\) on \(M\times S^1\) to the universal covering \(M\times \mathbb{R}\) and then sending it to its Ricci soul.

  To see that this yields a retraction, note first that by the dimension assumption and the \(h\)-cobordism theorem, every Ricci soul of a non-negatively Ricci curved metric on \(M\times \R\) is diffeomorphic to \(M\). Thus the map sending a metric on \(M\times S^1\) to its Ricci soul is well-defined and continuous by the above theorem.
  A section to this map is given by sending a metric on \(M\) to the product metric with the circle of length \(2\pi\).
  Therefore the claim follows.
\end{proof}

The following corollary allows a concrete description of the moduli space of non-negatively Ricci curved metrics on \(N_0\times S^1\) to be given in terms of the space of non-negatively Ricci curved metrics on \(N_0\).

\begin{cor}
\label{sec:metrics-mtimes-s1-2}
Let \(M\) be a closed connected non-negatively Ricci curved manifold with fundamental group \(\pi_1(M)\cong \mathbb{Z}\).
By Lemma~\ref{sec:metrics-mtimes-s1}, \(M\) is then a fiber bundle over \(S^1\) with closed simply connected fiber \(N_0\).

Let \(\mathcal{M}'_{\Ric\geq 0}(M)\) be the moduli space of non-negatively Ricci curved metrics whose Ricci soul is diffeomorphic to \(N_0\).
Then there is an embedding
  \begin{align*}
    \Phi:\mathcal{M}'_{\Ric \geq 0}(M)&\rightarrow \mathcal{R}_{\Ric\geq 0}(N_0)\times_{\Diff(N_0)}(\Diff(N_0)/\tau)\times \R_{>0},& [g]&\mapsto ([h,\varphi(1)],a),
  \end{align*}
where \(h\) and \(\varphi\) are as in Lemma \ref{sec:metrics-mtimes-s1}. Here 
\(\tau\) denotes the involution on \(\Diff(N_0)\) which sends a diffeomorphism to its inverse.
Moreover, \(\Diff(N_0)\) acts on on \(\Diff(N_0)/\tau\) by conjugation
and \(a\) is the length of the circle \(F\) as defined in Lemma~\ref{sec:metrics-mtimes-s1}.

This embedding identifies \(\mathcal{M}'_{\Ric \geq 0}(M)\) with its image
\[A=\{([g,\varphi],a);\; g\in \mathcal{R}_{\Ric \geq 0}(N_0), \varphi\in \Iso(g)\cap B, a>0\},\]
where \(B\) is the set of those diffeomorphisms of \(N_0\) whose mapping torus is diffeomorphic to \(M\).
\end{cor}

\

\begin{proof}
  To prove that \(\Phi\) is continuous, it suffices to show that the map
  \begin{align*}
    \mathcal{R}'_{\Ric \geq 0}(M)&\rightarrow \mathcal{R}_{\Ric\geq 0}(N_0)\times_{\Diff(N_0)}(\Diff(N_0)/\tau)& g&\mapsto [h,\varphi(1)]    
  \end{align*}
  is well-defined, continuous and \(\Diff(M)\)-invariant.
Here \(\mathcal{R}'_{\Ric \geq 0}(M)\) denotes the space of metrics on \(M\) whose Ricci soul is diffeomorphic to \(N_0\).

To prove well-definedness, 
let \(h_0\) and \(h_1\) be two metrics on \(N_0\) and \(\varphi_0\) and \(\varphi_1\)  be the corresponding elements of the diffeomorphism groups of \(N_0\) such that \(M\) is of the form described in Lemma \ref{sec:metrics-mtimes-s1}  for both pairs \((h_0,\varphi_0)\) and \((h_1,\varphi_1)\).
Then we have isometries which are compatible with the deck transformation groups
\begin{equation*}
  \tilde{\psi}_i:(N_0,h_i)\times \R \rightarrow \tilde{M}
\end{equation*}

Since the isometry \(\phi=\tilde{\psi}_0^{-1}\circ \tilde{\psi}_1\) maps lines to lines and \(N_0\) is perpendicular to all lines in \((N_0,h_i)\times \R\),
it follows that \(\phi\) is of the form \(\phi_0\times \pm\Id_\R\) where \(\phi_0:(N_0,h_0)\rightarrow (N_0,h_1)\) is an isometry.

Since the deck transformation groups in \((N_0,h_i)\times \R\rightarrow M\), \(i=1,2\) are conjugate via \(\phi\),
it follows that the \(\varphi_i\) are conjugate via \(\phi_0\).
This shows that the map \(\Phi\) is also well-defined.

Moreover, the \(\Diff(M)\) invariance is clear by the following argument: 
if \((h_0,\varphi_0)\) are as in Lemma \ref{sec:metrics-mtimes-s1} for the metric \(g\) on \(M\),
then for any diffeomorphism \(\phi\) of \(M\)  one has that \((h_0,\varphi_0)\) 
will also work in Lemma~\ref{sec:metrics-mtimes-s1} 
for the metric \(\phi^*g\).

Therefore it remains to show the continuity.
To do so, let \(g_0\) be a metric on \(M\) and \(g_n\) be a sequence of metrics on \(M\) converging to \(g_0\).
By the above theorem, we can find sequences \(h_n\) of metrics on \(N_0\) and 
diffeomorphisms  \(\varphi_n\) of \(N_0\) such that the following holds:
\begin{itemize}
\item the  \((h_n,\varphi_n)\) are as in Lemma \ref{sec:metrics-mtimes-s1} for \((M,g_n)\);
\item the \(h_n\) converge to \(h_0\);
\item there are isometries \(\alpha_n\) of \((N_0,h_0)\) such that \(\alpha_n\circ \varphi_0\circ \alpha_n^{-1}\) is close to \(\varphi_n\) in \(\Diff(N_0)\) for \(n\) large.
\end{itemize}
Since, moreover, \begin{equation*}
[h_0,\varphi_0]=[\alpha_n^*h_0,\varphi_0]=[h_0,\alpha_n\circ\varphi_0\circ\alpha_n^{-1}],
\end{equation*}
in the target of our map \(\Phi\) the continuity follows.

It is clear that \(A\) is the image of our map \(\Phi\).
Moreover, it follows from Lemma \ref{sec:metrics-mtimes-s1} that \(\Phi\) is also injective.
Indeed, if \(\Phi([g_1])=\Phi([g_2])=([h,\varphi],a)\), then there are Riemannian normal coverings \(\beta_i:(N_0,h)\times \R\rightarrow (M,g_i)\) for \(i=1,2\).
Therefore we have to show that the deck transformation groups \(\Gamma_i\) of these coverings are conjugate in \(\Iso(N_0,h)\times \Iso(\R)\).
Let \(\tau_i\) be generators of these groups.
Then we have \(\tau_i=(\varphi^{\pm 1}, f_{\pm a})\), where \(f_{\pm a}=f_{\mp a}^{-1}\) is the translation by \(a\) in the \(\R\)-factor.
Since \((\Id_{N_0},\pm \Id_{\R})\in \Iso(N_0,h)\times \Iso(\R)\), it follows that \(\tau_1\) is conjugate to \(\tau_2\) or \(\tau_2^{-1}\).
In particular \(\Gamma_1\) and \(\Gamma_2\) are conjugate in \(\Iso(N,h)\times \Iso(\R)\) and the injectivity follows.

Therefore, it only remains to show the continuity of the inverse map \(\Phi^{-1}\) on \(A\).

Denote by

\[\pi:\mathcal{R}(N_0)\times \Diff(N_0)\times \R_{>0}\rightarrow \mathcal{R}(N_0)\times_{\Diff(N_0)}(\Diff(N_0)/\tau)\times \mathbb{R}_{>0}\]
the projection.

Let \((h_n,\varphi_n,a_n)\in \pi^{-1}(A)\) converge to \((h_0,\varphi_0,a_0)\).
Moreover, let \(\delta_n\) be a diffeomorphism from \(M\) to the mapping torus  \(N_0\times [0,a_n]/\sim_n\) of \(\varphi_n\).
Here we have \((x,0)\sim_n(\varphi_n(x),a_n)\) for \(x\in N_0\).

Let \(\tilde{h}_n\) be the metric on \(N_0\times_{a_n\mathbb{Z}} \mathbb{R}\) induced from the product metric \(h_n+dt^2\) on \(N_0\times \mathbb{R}\). Here \(a_n\in a_n\mathbb{Z}\) acts as \(\varphi_n\) on \(N_0\) and as translation by \(a_n\) on \(\mathbb{R}\).

The inclusion \(N_0\times [0,a_n]\hookrightarrow N_0\times \mathbb{R}\) induces a natural diffeomorphism from  the mapping torus  of \(\varphi_n\)  to \(N_0\times_{a_n\mathbb{Z}} \mathbb{R}\).
So we can use this diffeomorphism to identify these two manifolds.

Let \(g_n=\delta_n^*\tilde{h}_n\) be the induced metric on \(M\).

For \(n\) large enough, we can find a path \(\phi_{n,t}\) from the identity in \(\Diff(N_0)\) to \(\varphi_0\circ \varphi_n^{-1}\).
Denote by \(\phi_n\) the diffeomorphism
\begin{align*}
  N_0\times[0,a_n]/\sim_n &\rightarrow N_0\times [0,a_0]/\sim_0,& (x,t)&\mapsto (\phi_{n,t}(x),t)
\end{align*}
 induced by \(\phi_{n,t}\) on mapping tori.
Since \((\varphi_n,a_n)\) converges to \((\varphi_0,a_0)\), we can assume that
 the lift of \(\phi_n\) to \(N_0\times[0,a_n]\) is close to the identity.
 Since the \(h_n\) converge to \(h_0\), it also follows that \((\phi_n^{-1})^*\tilde{h}_n\) is close to \(\tilde{h}_0\) for large \(n\).

 Therefore \[(\delta_n^{-1}\circ \phi_n^{-1}\circ \delta_0)^*g_n=\delta_0^*\circ(\phi_n^{-1})^*\tilde{h}_n\]
 is close to \(g_0=\delta_0^*\tilde{h}_0\) for large \(n\).
 Hence it follows that the assignment \((h_n,\varphi_n,a_n)\mapsto [g_n]\) is continuous.
 Because this assignment induces the map \(\Phi^{-1}\) on \(A\) and \(A\) is equipped with the quotient topology, it follows that \(\Phi^{-1}\) in continuous.

Therefore the proof of the corollary is complete.
\end{proof}

\section{Path components and rational cohomology groups of products with flat tori}
\label{sec:products-with-higher}

In this section we investigate the components and rational cohomology groups 
of moduli spaces of non-negative Ricci curved metrics on
products of manifolds with tori. To this means, we first generalize Theorem~\ref{sec:metrics-mtimes-r} as follows:

\

\begin{theorem}
\label{sec:metrics-mtimes-t}
  Let \((M,g_0)\) be isometric to \((N_0,h_0)\times \R^k\) with \((N_0,h_0)\) a 
 compact Riemannian manifold and \(\mathbb{R}^k\) equipped with the standard flat metric. Assume that there are metrics \(g_n\) on \(M\) which converge in the \(C^r\)-topology, \(r> 1\), to \(g_0\) uniformly on compact subsets and 
  assume furthermore that \((M,g_n)\) is isometric to \((N_n,h_n)\times \R^k\) with \(N_n\) compact and \(\mathbb{R}^k\) as above.

Then for large enough \(n\in \N\) the following is true:
\begin{itemize}
  \item there are diffeomorphisms \(\phi_n: N_0 \rightarrow N_n\), such that
  \item the metrics \(\phi_n^*h_n\) converge in the \(C^r\)-topology to \(h_0\).
  \end{itemize}
\end{theorem}

\

\begin{proof}
Let \(\psi_n:(M,g_n)\rightarrow (N_n,h_n)\times \R^k\) be isometries.
For each \(n\), let \(Y_{in}=\psi_n^*\frac{\partial}{\partial x_i}\), \(i=1,\dots,k\), where the \(\frac{\partial}{\partial x_i}\) denote the standard coordinate vector fields on \(\mathbb{R}^k\).
Then the \(Y_{in}\) are parallel vector fields on \(M\) whose integral curves are lines in \((M,g_n)\).
Moreover, \(Y_{1n}(x),\dots,Y_{kn}(x)\) form an orthonormal basis of \(D_{\psi_n(x)}\psi_n^{-1}(\R^k)\) for each \(x\in M\).

Since lines in \((M,g_n)\) converge to lines in \((M,g_0)\),
for each \(i\) and each subsequence \(Y_{in_l}\) there is a subsequence \(Y_{in_{l_j}}\) which converges to some vector field \(Y_{i0}\) whose integral curves are lines.
The \(Y_{i0}\) might still depend on the subsequences. But as we will show next, their span does not. Their common span is equal to  \( D\psi_0^{-1}(\mathbb{R}^k)\).

Indeed, since the integral curves of the \(Y_{i0}\) are lines, we have \(Y_{i0}(x)\in D\psi_0^{-1}(\mathbb{R}^k)\) for all \(x\in M\) and \(i=1,\dots,k\).
Moreover, because the \(Y_{1n_{l_j}}(x),\dots,Y_{kn_{l_j}}(x)\) are orthonormal and the whole sequence of the \(g_n\) converges to \(g_0\), the \(Y_{10}(x),\dots,Y_{k0}(x)\) form an orthonormal basis of \(D_{\psi_0(x)}\psi_0^{-1}(\R^k)\).

Therefore the orthogonal projections \(T_xM\rightarrow D_{\psi_n(x)}\psi_n^{-1}(\R^k)\) converge to the orthogonal projection \(T_xM\rightarrow D_{\psi_0(x)}\psi_0^{-1}(\R^k)\).

Let us now identify each \(N_n\) with the image under the inclusion map

\begin{equation*}
N_n=N_{n}\times \{0\}\hookrightarrow N_n\times \mathbb{R}^k\rightarrow_{\psi_n}M.  
\end{equation*}

This identifies \(T_pN_n\) with \( \langle Y_{1n}(\psi_n^{-1}(p,0)),\dots,Y_{kn}(\psi_n^{-1}(p,0)) \rangle^{\perp_{g_n}} \subset T_{\psi_n^{-1}(p,0)} M\).

Moreover, the differential of the projection
\begin{equation*}
  \psi_n':M\rightarrow_{\psi_n^{-1}}N_n\times \mathbb{R}^k\rightarrow N_n
\end{equation*}
is then just the orthogonal projection
\[d\psi_n'(v)=v-\sum_{i=1}^kg_n(v,Y_{in})Y_{in}.\]
This means that,  for sufficiently large \(n\), \(\phi_n=\psi'_n\circ \psi_0\) is an immersion and therefore a covering.
Since \(\phi_n\) induces an isomorphism on fundamental groups, it follows that \(\phi_n\) is a diffeomorphism.
A similar computation as in the proof of Theorem~\ref{sec:metrics-mtimes-r} then shows that the metrics \(\phi_n^*h_n\) converge in \(C^r\)-topology to \(h_0\).

Indeed, if \(X,Y,Z\) are vector fields on \(N_0\), we can compute
\begin{align*}
  X(\phi_n^*h_n(Y,Z))&=Xg_n(Y,Z)-\sum_{i=1}^kX(g_n(Y,Y_{in})g_n(Z,Y_{in}))\\
                     &= Xg_n(Y,Z)-\sum_{i=1}^k\big[(g_n(\nabla^n_X Y,Y_{in})+g_n(Y,\nabla^n_X Y_{in}))g_n(Z,Y_{in})\\
                     &+(g_n(\nabla^n_X Z,Y_{in})+g_n(Z,\nabla^n_X Y_{in}))g_n(Y,Y_{in}))\big]\\
                     &=Xg_n(Y,Z)-\sum_{i=1}^k\big[g_n(\nabla^n_X Y,Y_{in})g_n(Z,Y_{in})+g_n(\nabla^n_X Z,Y_{in})g_n(Y,Y_{in})\big]\\
                     &\rightarrow Xg_0(Y,Z)=Xh_0(Y,Z) \text{ as } n\rightarrow \infty.
\end{align*}

Here \(\nabla^n\) denotes the Riemannian connection on \((M,g_n)\).
Note here that under the isometry \((M,g_n)\cong (N_n\times \R^k,h_n+ \sum_{i=1}^kdx_i^2)\), the vector field \(Y_{in}\) corresponds to a constant length vector field  on \(N_n\times \R^k\) tangent to the second factor.
Hence, \(\nabla^n Y_{in}=0\) follows.

Moreover, to see that \(g_n(\nabla^n_X Y,Y_{in})g_n(Z,Y_{in})\) converges to zero as \(n\) grows to infinity one can argue as follows.
It suffices to show that for every sequence \(n_k\in\mathbb{N}\) with \(\lim_{k\to \infty}n_k=\infty\) there is subsequence \(n_{k_l}\) such that
\[\lim_{l\to\infty}g_{n_{k_l}}(\nabla^{n_{k_l}}_X Y,Y_{i{n_{k_l}}})g_{n_{k_l}}(Z,Y_{i{n_{k_l}}})=0.\]
  Notice that we already know that there is a subsequence \(n_{k_l}\) such that \(Y_{in_{k_l}}\) converges to a parallel vector field tangent to the \(\mathbb{R}^k\)-factor.
  Moreover, \(\nabla^{n_{k_l}}_X Y\) converges to \(\nabla^0_XY\) which is tangent to the \(N_0\)-factor of \(M\), because \(N_0\) is totally geodesic in \(M\).
  Since \(g_n\) converges to \(g_0\) the claim follows.

Notice, moreover, that the above convergence is uniform on \(N_0\).

 We can deal with higher derivatives of the metrics in a similar way, and hence the claim follows.
\end{proof}

Now we are in the position to generalize Corollary \ref{sec:metr-prod-with-1} to products \(M\times T^k\) with tori of dimension greater than one.
To do this we first state the following consequence of the s-cobordism theorem.

\begin{lemma}
  \label{sec:path-comp-rati-1}
  Let \(M_1\) and \(M_2\) be simply connected closed manifolds of dimension at least \(5\) and \(E_i\rightarrow T^k\), \(i=1,2\), be \(M_i\)-bundles such that \(E_1\) and \(E_2\) are diffeomorphic.
  Then \(M_1\) and \(M_2\) are diffeomorphic.
\end{lemma}
\begin{proof}
  We prove this lemma by induction on \(k\). If \(k=0\) there is nothing to show.
  Let us therefore assume \(k\geq 1\) and that the claim is proved for \((k-1)\)-dimensional tori \(T^{k-1}\).

  Note that \(\pi_1(E_1)=\pi_1(E_2)=\pi_1(T^k)=\mathbb{Z}^{k-1}\oplus \mathbb{Z}\).
  Consider the coverings \(\tilde{E}_1\) and \(\tilde{E}_2\) of \(E_1\) and \(E_2\) corresponding to
  the  \(\mathbb{Z}^{k-1}\)-summand of \(\pi_1(E_i)\).
  Then \(\tilde{E}_i\) is diffeomorphic to \(E_i'\times \mathbb{R}\), where \(E_i'\) is an \(M_i\)-bundle over \(T^{k-1}\).
  Moreover, \(\tilde{E}_1\) is diffeomorphic to \(\tilde{E}_2\).

  Note now that the Whitehead torsion of the free abelian group \(\pi_1(E_i')=\mathbb{Z}^{k-1}\) vanishes.
  Therefore it follows from the s-cobordism theorem that \(E_1'\) and \(E_2'\) are diffeomorphic.
  Hence the claim follows from the induction hypothesis.
\end{proof}

\begin{cor}
\label{sec:products-with-higher-1}
  Let \(M\) be a simply connected closed manifold of dimension at least \(5\)
  that admits a metric with non-negative Ricci curvature.
  Then there is a retraction
  \begin{equation*}
    \mathcal{M}_{Ric\geq 0}(M\times T^k)\rightarrow \mathcal{M}_{\Ric\geq 0}(M),
  \end{equation*}
 In particular, we have that the moduli space \(\mathcal{M}_{\Ric\geq 0}(M\times T^k)\) has at least as many path components as \(\mathcal{M}_{\Ric \geq 0}(M)\).
\end{cor}
\begin{proof}
  Note first that by Lemma~\ref{sec:metrics-mtimes-s1} and Lemma~\ref{sec:path-comp-rati-1} the Ricci soul of any non-negatively Ricci curved metric on \(M\times T^k\) is diffeomorphic to \(M\).
  
  By Theorem~\ref{sec:metrics-mtimes-t}, the map which sends a metric on \(M\times T^k\) to its Ricci soul is continuous and well-defined.
A section to this map is given by sending a metric on \(M\) to the product metric with a product of \(k\) circles of length \(2\pi\).
Hence the claim follows.
\end{proof}

\begin{theorem}
\label{sec:products-with-higher-2}
  Let \(M\) be a non-negatively Ricci-curved manifold as in Lemma~\ref{sec:metrics-mtimes-s1}.
  Then there is a continuous map
  \begin{align*}
    \mathcal{M}_{\Ric\geq 0}(M)&\rightarrow \mathcal{M}_{\seccur=0}(F)& [g=h'\oplus_s h]&\mapsto [h],
  \end{align*}
  where \(h'\) and \(h\) are the metrics on the fiber and base of the Riemannian submersion \((M,g)\rightarrow F\) as in Lemma~\ref{sec:metrics-mtimes-s1}.

 Moreover, if \(M\) is diffeomorphic to \(N\times F\) for some non-negatively Ricci curved manifold \(N\), then the above map has a section.
  Therefore it is a retraction.
\end{theorem}
\begin{proof}
  Note that convergence in the moduli space means convergence up to isometries.
  
  Let \([g_n]\in \mathcal{M}_{\Ric\geq 0}(M)\) converge to \([g_0]\).
  By Lemma \ref{sec:metrics-mtimes-s1},
  \((M,g_n)\) is isometric to the total space of a fiber bundle with base space  a flat manifold \((F,h_n)\).
  Moreover, by construction, the isometry class of \(h_n\) depends only on the isometry class of \(g_n\).
 
  Note that Gromov-Hausdorff convergence of flat metrics on \(F\) coincides with convergence in \(C^{\infty}\)-topology. 
  To prove the first claim, it thus suffices to show that the isometry classes of the metrics \(h_n\) on \(F\) Gromov-Hausdorff converge to the isometry class of \(h_0\).
  If we shrink the fibers of the Riemannian submersions \((M,g_n)\rightarrow (F,h_n)\), then the isometry classes of the shrunk metrics \((M,g_{n,m})\) Gromov-Hausdorff converge to the one of \((F,h_n)\).
  Thus \(\lim_{m\to \infty} [M,g_{n,m}]=[F,h_n]\).
  Now  \(\lim_{n\to\infty}[F,h_n]=[F,h_0]\) follows as in the proof of Theorem~\ref{sec:metrics-mtimes-s1-1} by interchanging the limits.
  Hence the first claim is proved.

  If \(M\) is diffeomorphic to \(N\times F\), a section is given by sending a metric on \(F\) to its product metric with some fixed non-negatively Ricci-curved metric on \(N\).
\end{proof}

\

\begin{prop}
\label{sec:path-comp-rati}
  If \(T^n\) is an \(n\)-dimensional standard torus, then the following holds:
  \begin{enumerate}
  \item The moduli space \(\mathcal{M}_{\sec=0}(T^n)\) of flat metrics on \(T^n\)  is simply connected.
  \item If \(n=1,2,3\), then \(\mathcal{M}_{\sec=0}(T^n)\) is contractible.
  \item If \(n=4\), then \(\pi_3(\mathcal{M}_{\sec=0}(T^n))\otimes \mathbb{Q}\cong \mathbb{Q}\).
  \item If \(n>4\), and \(n\neq 8,9,10\), then \(\pi_5(\mathcal{M}_{\sec=0}(T^n))\otimes \mathbb{Q}\cong \mathbb{Q}\).
  \end{enumerate}
\end{prop}

\begin{proof}
  By \cite{MR0322748}, the space \(\mathcal{M}=\mathcal{M}_{\sec=0}(T^n)\)  is homeomorphic to the biquotient \[O(n)\backslash GL(n,\R)/GL(n,\mathbb{Z}).\]

The space \(O(n)\backslash GL(n,\R)\) of scalar products on \(\R^n\) is a model for the classifying space \(E_{\mathcal{FIN}} GL(n,\mathbb{Z})\) for the family \(\mathcal{FIN}\)  of finite subgroups of \(GL(n,\mathbb{Z})\).
 Hence it follows from Lemma 4.14 of \cite{MR2361088} that \(\mathcal{M}\) has the same rational cohomology groups as the classifying space \(B GL(n,\mathbb{Z})\).

The rational cohomology groups \(H^3(BGL(n,\mathbb{Z});\mathbb{Q})\) and \(H^5(BGL(n,\mathbb{Z});\mathbb{Q})\) of this space have been computed for \(n=4\) \cite{MR0491893}, for \(n=5,6,7\) \cite{MR3084439}, and for \(n\geq 11\) \cite{MR0387496}, \cite{MR586429}.
In the first case the first group is isomorphic to \(\mathbb{Q}\), whereas in the other cases the second group is isomorphic to \(\mathbb{Q}\).
Note that in all these cases these groups are the first non-trivial rational cohomology groups.

Hence the last two claims follow from the Hurewicz theorem once we have proved the first claim.

The second claim for \(n=1\) is trivial because \(GL(1,\mathbb{Z})\) is a finite group.
For \(n=3\) this claim was proven by Soul\'e \cite{MR0470141}. For \(n=2\) one can argue as follows.

 Taking the first claim for granted, \(O(n)\backslash GL(n,\R)_\pm/GL(n,\mathbb{Z})\) is a two-dimensional simply connected model for 
 \[B_{\mathcal{FIN}}GL(2,\mathbb{Z})=(E_{\mathcal{FIN}}GL(2,\mathbb{Z}))/GL(2,\mathbb{Z}).\]
 Here \(GL(2,\mathbb{R})_\pm\) is the subgroup of \(GL(2,\mathbb{R})\) consisting of those \(n\times n\)-matrices \(A\) with \(\det A=\pm 1\).

 Hence the only possibly non-trivial homology group of this space with integer coefficients is \(H_2\).
 Moreover, \(H_2\) must be torsion-free.
 But it is a classical result (see for example \cite{MR1954121}), that \[H_2(B_{\mathcal{FIN}}GL(2,\mathbb{Z});\mathbb{Q})\cong H_2(BGL(2,\mathbb{Z});\mathbb{Q})\] is the trivial group.
 Therefore this also holds with integer coefficients and the claim follows.

Hence it remains to prove that \(O(n)\backslash GL(n,\R)/GL(n,\mathbb{Z})\) is simply connected for \(n\geq 2\),
and this can be done as follows:

From \cite{armstrong} we obtain that
the fundamental group
 \[\pi_1\,(O(n)\backslash GL(n,\R)/GL(n,\mathbb{Z}))\] is isomorphic to \(GL(n,\mathbb{Z})/N\), where \(N\) is the normal subgroup of \(GL(n,\mathbb{Z})\) generated by those elements which have fixed points in \(O(n)\backslash GL(n,\R)\).
By the above discussion, it follows that the group  \(N\) is the subgroup of \(GL(n,\mathbb{Z})\) which is generated by all elements of finite order.

By Corollary 16.3 from \cite{MR0349811} and the remark following it, 
the group \(SL(n,\mathbb{Z})\) is generated by elementary matrices if \(n\geq 3\). Hence, it follows that \(SL(n,\mathbb{Z})\) is generated by the images of finitely many group homomorphisms \(SL(2,\mathbb{Z})\rightarrow SL(n,\mathbb{Z})\).
Since \(SL(2,\mathbb{Z})\) is generated by two elements of finite order, it follows that \(SL(n,\mathbb{Z})\) is generated by finitely many elements of finite order for every \(n\geq 1\).

Because \(GL(n,\mathbb{Z})\) is a semi-direct product of \(SL(n,\mathbb{Z})\) and \(\mathbb{Z}/2\), this implies in turn  that \(GL(n,\mathbb{Z})=N\). Hence the biquotient space
\[O(n)\backslash GL(n,\R)/GL(n,\mathbb{Z})\]
is simply connected. 
\end{proof}

\

\begin{cor}
\label{sec:products-with-higher-3}
 Let \(M\) be a simply connected closed smooth manifold which admits a metric
  with non-negative Ricci curvature, and let  \(T\) be a torus of dimension \(n\geq 4\), \(n\neq 8,9,10\).
  Then the moduli space \(\mathcal{M}_{\Ric\geq 0}(M\times T)\) of non-negatively Ricci curved metrics on $M\times T$
  has non-trivial higher rational cohomology groups and non-trivial higher rational homotopy groups.
\end{cor}

\begin{proof}
  By Theorem \ref{sec:products-with-higher-2}, the map  \(\mathcal{M}_{\Ric\geq 0}(M\times T)\rightarrow \mathcal{M}_{\seccur=0}(T)\) is a retraction.
  Therefore the cohomology of \(\mathcal{M}_{\Ric\geq 0}(M\times T)\) is at least as complicated as the cohomology of \(\mathcal{M}_{\seccur=0}(T)\).
  Hence the claim follows from Proposition \ref{sec:path-comp-rati} above.
\end{proof}

\normalsize

~\\
Wilderich Tuschmann\\ Karlsruher Institut f\"ur Technologie,
Institut f\"ur Algebra und Geometrie, Arbeitsgruppe Differentialgeometrie\\
Englerstr. 2, D-76131 Karlsruhe, Germany\\
\texttt{tuschmann@kit.edu}

~\\Michael Wiemeler\\Mathematisches Institut, WWU M\"unster\\ Einsteinstr. 62, D-48149 M\"unster, Germany\\
\texttt{wiemelerm@uni-muenster.de}
\end{document}